\documentclass{amsart}

\usepackage{amsmath}
\usepackage{amsfonts}
\usepackage{amssymb,enumerate}
\usepackage{amsthm}
\usepackage[all]{xy}
\usepackage{hyperref}

\newtheorem{lem}{Lemma}[section]

\newtheorem{prop}[lem]{Proposition}
\newtheorem{thm}[lem]{Theorem}

\newtheorem{Defn}[lem]{Definition}
\newtheorem{Ex}[lem]{Example}
\newtheorem{Question}[lem]{Question}
\newtheorem{Property}[lem]{Property}
\newtheorem{Properties}[lem]{Properties}
\newtheorem{Discussion}[lem]{Remark}
\newtheorem{Construction}[lem]{Construction}
\newtheorem{Notation}[lem]{Notation}
\newtheorem{Fact}[lem]{Fact}
\newtheorem{Outline}[lem]{Outline}
\newtheorem{Assumption}[lem]{Assumption}
\newtheorem{Notationdefinition}[lem]{Definition/Notation}
\newtheorem{Remarkdefinition}[lem]{Remark/Definition}
\newtheorem{Subprops}{}[lem]
\newtheorem{Para}[lem]{}
\newtheorem{Step}[lem]{Step}

\newenvironment{defn}{\begin{Defn}\rm}{\end{Defn}}
\newenvironment{ex}{\begin{Ex}\rm}{\end{Ex}}

\newenvironment{fact}{\begin{Fact}\rm}{\end{Fact}}
\newenvironment{outline}{\begin{Outline}\rm}{\end{Outline}}

\newenvironment{disc}{\begin{Discussion}\rm}{\end{Discussion}}

\newenvironment{step}{\begin{Step}\rm}{\end{Step}}

\newcommand{\catd}{\mathsf{D}}

\newcommand{\pd}{\operatorname{pd}}	
\newcommand{\gdim}{\mathrm{G}\text{-}\!\dim}

\newcommand{\id}{\operatorname{id}}

\newcommand{\depth}{\operatorname{depth}}

\newcommand{\ann}{\operatorname{Ann}}

\newcommand{\type}{\operatorname{type}}

\newcommand{\ext}{\operatorname{Ext}}	
\newcommand{\rhom}{\mathbf{R}\!\operatorname{Hom}}	

\newcommand{\HH}{\operatorname{H}}
\newcommand{\Hom}{\operatorname{Hom}}	
\newcommand{\coker}{\operatorname{Coker}}
\newcommand{\spec}{\operatorname{Spec}}

\newcommand{\tor}{\operatorname{Tor}}

\newcommand{\shift}{\mathsf{\Sigma}}

\newcommand{\Ker}{\operatorname{Ker}}

\newcommand{\tatetor}{\widehat{\tor}}
\newcommand{\tateext}{\widehat{\ext}}

\newcommand{\ideal}[1]{\mathfrak{#1}}
\newcommand{\m}{\ideal{m}}

\newcommand{\ol}{\overline}

\newcommand{\bbz}{\mathbb{Z}}

\newcommand{\xra}{\xrightarrow}

\newcommand{\y}{\mathbf{y}}

\newcommand{\x}{\mathbf{x}}

\newcommand{\cd}{C^{\dagger}}

\renewcommand{\geq}{\geqslant}
\renewcommand{\leq}{\leqslant}
\renewcommand{\ker}{\Ker}
\renewcommand{\hom}{\Hom}

\numberwithin{equation}{lem}

\begin{document}

\bibliographystyle{amsplain}

\author[Jorgensen]{David A. Jorgensen}

\address{David A. Jorgensen, Department of Mathematics,
University of Texas at Arlington,
Arlington, TX  76019 USA}

\email{djorgens@uta.edu}

\urladdr{http://dreadnought.uta.edu/\~{}dave/}

\author[Leuschke]{Graham J. Leuschke}

\address{Graham J. Leuschke, Mathematics Department,
215 Carnegie Hall,
Syracuse University,
Syracuse, NY 13244 USA}

\email{gjleusch@math.syr.edu}

\urladdr{http://www.leuschke.org/}

\author[Sather-Wagstaff]{Sean Sather-Wagstaff}

\address{Sean Sather-Wagstaff,
Department of Mathematics,
NDSU Dept \# 2750,
PO Box 6050,
Fargo, ND 58108-6050
USA }

\email{Sean.Sather-Wagstaff@ndsu.edu}

\urladdr{http://www.ndsu.edu/pubweb/\~{}ssatherw/}

\thanks{D.\ Jorgensen and S.\ Sather-Wagstaff were partly supported by NSA grants.
G.\ Leuschke was partly supported by NSF grant DMS 0556181.
}

\title{Presentations of rings with non-trivial semidualizing modules}




\begin{abstract}
Let $R$ be a commutative noetherian local ring.
A finitely generated $R$-module $C$ is \emph{semidualizing} if it is self-orthogonal and satisfies the condition $\Hom_R(C,C)\cong R$.
We prove that a Cohen-Macaulay ring $R$ with dualizing module $D$ admits a semidualizing module $C$ satisfying $R\ncong C \ncong D$
if and only if it is a homomorphic image of a Gorenstein ring in which the defining ideal decomposes in a cohomologically independent way. This expands on a well-known result of Foxby, Reiten and Sharp saying that $R$ admits a dualizing module if and only if
$R$ is Cohen--Macaulay and a homomorphic image of a local Gorenstein ring.
\end{abstract}

\maketitle

\section{Introduction} \label{sec00}

Throughout this paper
$(R,\m,k)$ is a commutative noetherian local ring.

A finitely generated $R$-module $C$ is
\emph{self-orthogonal} if $\ext^i_R(C,C)=0$ for all $i\geq 1$.
Examples of self-orthogonal $R$-modules include the finitely generated free $R$-modules
and the dualizing module of
Grothendieck.
(See Section~\ref{sec01}
for definitions and background information.)
Results of
Foxby~\cite{foxby:gmarm},
Reiten~\cite{reiten:ctsgm}
and Sharp~\cite{sharp:gmccmlr}
precisely characterize the local rings which possess a dualizing module:
the ring $R$ admits a dualizing module if and only if
$R$ is Cohen--Macaulay
and there exist a Gorenstein local ring $Q$ and an ideal $I\subset Q$
such that $R\cong Q/I$.

The point of this paper is to similarly characterize the local
Cohen--Macaulay rings with a dualizing module which admit certain other self-orthogonal modules.
The specific self-orthogonal modules of interest are the \emph{semidualizing}
$R$-modules, that is, those self-orthogonal $R$-modules
satisfying $\hom_R(C,C)\cong R$. A free $R$-module of rank 1
is semidualizing, as is a dualizing $R$-module, when one exists.
We say that a semidualizing is \emph{non-trivial} if it is neither free
nor dualizing.

Our main theorem is the following expansion of the aforementioned
result of Foxby, Reiten and Sharp; we prove it in Section~\ref{sec02}.
It shows, assuming the existence of a dualizing module, that $R$ has a non-trivial semidualizing module if and only if $R$ is Cohen-Macaulay and
$R\cong Q/(I_1+I_2)$ where $Q$ is Gorenstein and the rings
$Q/I_1$ and $Q/I_2$ enjoy  considerable
cohomological vanishing over $Q$.
Thus, it addresses both of  the following questions:
what conditions guarantee that $R$ admits a non-trivial semidualizing module, and
what are the ramifications of the existence of such a module?

\begin{thm} \label{thm0001}
Let $R$ be a local Cohen--Macaulay ring with a dualizing module.
Then  $R$ admits a semidualizing module that is neither dualizing nor free
if and only if there exist
a Gorenstein local ring $Q$ and
ideals $I_1,I_2\subset Q$ satisfying the following conditions:
\begin{enumerate}[\quad\rm(1)]
\item \label{thm0001a}
There is a ring isomorphism
$R\cong Q/(I_1+I_2)$;
\item \label{thm0001b}
For $j=1,2$ the quotient ring $Q/I_j$ is Cohen--Macaulay
and not Gorenstein;
\item \label{thm0001d}
For all $i\in\bbz$, we have the following vanishing of Tate cohomology modules:
$\tatetor^Q_i(Q/I_1,Q/I_2)=0=\tateext^i_Q(Q/I_1,Q/I_2)$;
\item \label{thm0001e}
There exists an integer $c$ such that $\ext^c_Q(Q/I_1,Q/I_2)$ is not cyclic; and
\item \label{thm0001c}
For all $i\geq 1$, we have $\tor^Q_i(Q/I_1,Q/I_2)=0$; in particular,
there is an equality $I_1\cap I_2=I_1I_2$.
\end{enumerate}
\end{thm}

A prototypical example of a ring admitting non-trivial semidualizing modules is the following.

\begin{ex} \label{vasex}
Let $k$ be a field and set $Q=k[\![X,Y,S,T]\!]$. The ring
$$R=
Q/(X^2,XY,Y^2,S^2,ST,T^2)
=Q/[(X^2,XY,Y^2)+(S^2,ST,T^2)]
$$
is local with maximal ideal $(X,Y,S,T)R$.
It is artinian of socle dimension 4, hence Cohen--Macaulay and non-Gorenstein.
With $R_1=Q/(X^2,XY,Y^2)$
it follows that the $R$-module $\ext^2_{R_1}(R,R_1)$
is semidualizing and neither dualizing nor free;
see~\cite[p.~92, Example]{vasconcelos:dtmc}.
\end{ex}

Proposition~\ref{thm9901} shows how Theorem~\ref{thm0001}
can be used to construct numerous rings admitting non-trivial
semidualizing modules.  To complement this, the following example shows
that rings that do not admit non-trivial semidualizing modules are easy to come by.

\begin{ex} \label{intex}
Let $k$ be a field. The ring $R=k[X,Y]/(X^2,XY,Y^2)$
is local with maximal ideal $\m=(X,Y)R$.
It is artinian of socle dimension 2, hence Cohen--Macaulay and non-Gorenstein. 
From the equality $\m^2=0$, it is straightforward
to deduce that the only semidualizing $R$-modules, up to isomorphism,
are the ring itself and the dualizing module; see~\cite[Prop.\ (4.9)]{vasconcelos:dtmc}.
\end{ex}

\section{Background on Semidualizing Modules} \label{sec01}

We begin with relevant definitions.
The following notions were introduced independently
(with different terminology) by
Foxby~\cite{foxby:gmarm},
Golod~\cite{golod:gdagpi},
Grothendieck~\cite{grothendieck:tdfac, hartshorne:lc},
Vasconcelos~\cite{vasconcelos:dtmc} and
Wakamatsu~\cite{wakamatsu:mtse}.

\begin{defn} \label{d0201}
Let $C$ be an $R$-module. The \emph{homothety homomorphism}
is the map
$\chi^R_C\colon R\to\hom_R(C,C)$ given by
$\chi^R_C(r)(c)=rc$.

The $R$-module $C$ is \emph{semidualizing} if
it satisfies the following conditions:
\begin{enumerate}[\quad(1)]
\item  \label{d0201a}
The $R$-module $C$ is finitely generated;
\item  \label{d0201b}
The homothety map $\chi^R_C\colon R\to\hom_R(C,C)$, is an isomorphism; and
\item  \label{d0201c}
For all $i\geq 1$, we have $\ext^i_R(C,C)=0$.
\end{enumerate}
An $R$-module $D$ is \emph{dualizing} if it is semidualizing and
has finite injective dimension.
\end{defn}

Note that the $R$-module $R$ is semidualizing, so that every
local ring admits a semidualizing module.

\begin{fact} \label{f0206}
Let $C$ be a semidualizing $R$-module.
It is straightforward to show that a sequence
$\x=x_1,\ldots,x_n\in\m$ is $C$-regular if and only if it is $R$-regular.
In particular, we have
$\depth_R(C)=\depth(R)$; see, e.g., \cite[(1.4)]{sather:bnsc}. Thus, when $R$ is Cohen--Macaulay,
every semidualizing $R$-module is a maximal Cohen--Macaulay module.
On the other hand, if $R$ admits a dualizing module, then $R$ is Cohen--Macaulay
by~\cite[(8.9)]{sharp:gm}.
As $R$ is local, if it admits a dualizing module, then its
dualizing module is unique up to isomorphism; see, e.g.~\cite[(3.3.4(b))]{bruns:cmr}.
\end{fact}

The following definition and fact justify the term ``dualizing''.

\begin{defn} \label{d0204}
Let $C$ and $B$ be  $R$-modules.
The natural \emph{biduality homomorphism}
$\delta^B_C\colon C\to\hom_R(\hom_R(C,B),B)$
is given by
$\delta^B_C(c)(\phi)=\phi(c)$.
When $D$  is a dualizing $R$-module,
we set $\cd=\hom_R(C,D)$.
\end{defn}

\begin{fact} \label{f0203}
Assume that $R$ is Cohen--Macaulay with dualizing module $D$.
Let $C$ be a semidualizing $R$-module.
Fact~\ref{f0206} says that $C$ is a maximal Cohen--Macaulay $R$-module.
From standard duality theory, for all $i\neq 0$ we have
$$
\ext^i_R(C,D)=0=\ext^i_R(\cd,D)$$
and the natural biduality homomorphism
$\delta^D_C\colon C\to\hom_R(\cd,D)$ is an isomorphism;
see, e.g., \cite[(3.3.10)]{bruns:cmr}.
The $R$-module $\cd$ is  semidualizing
by~\cite[(2.12)]{christensen:scatac}. Also,
the evaluation map $C\otimes_R\cd\to D$ given by $c\otimes\phi\mapsto\phi(c)$
is an isomorphism, and one has $\tor^R_i(C,\cd)=0$ for all $i\geq 1$
by~\cite[(3.1)]{gerko:sdc}.
\end{fact}

The following construction is also known as the ``idealization'' of $M$.
It was popularized by Nagata,
but goes back at least to Hochschild~\cite{hochschild:cgaa},
and the idea behind the construction appears in work of Dorroh~\cite{dorroh}.
It is the key idea for the
proof of the converse of Sharp's result~\cite{sharp:gmccmlr} given by
Foxby~\cite{foxby:gmarm} and
Reiten~\cite{reiten:ctsgm}.

\begin{defn} \label{d0203}
Let $M$ be an $R$-module.
The \emph{trivial extension} of $R$ by $M$
is the ring $R\ltimes M$, described as follows.
As an additive abelian group, we have $R\ltimes M = R\oplus M$.
The multiplication in $R\ltimes M$ is given by the formula
$$(r,m)(r',m')=(rr',rm'+r'm).$$
The multiplicative identity on $R\ltimes M$ is $(1,0)$.
We let $\epsilon_M\colon R\to R\ltimes M$
and $\tau_M\colon R\ltimes M\to R$ denote the
natural injection and surjection, respectively.
\end{defn}

The next assertions are straightforward to verify.

\begin{fact} \label{f0208}
Let $M$ be an $R$-module.
The trivial extension $R\ltimes M$ is a commutative ring with identity.
The maps $\epsilon_M$ and $\tau_M$ are ring homomorphisms, and
$\ker(\tau_M)=0\oplus M$. We have
$(0\oplus M)^2=0$, and so $\spec(R\ltimes M)$ is in order-preserving bijection
with $\spec(R)$.
It follows that $R\ltimes M$ is quasilocal and $\dim(R\ltimes M)=\dim(R)$.
If $M$ is finitely generated, then $R\ltimes M$ is also noetherian and
$$\depth(R\ltimes M)=\depth_R(R\ltimes M)=\min\{\depth(R),\depth_R(M)\}.$$
In particular, if $R$ is Cohen--Macaulay and $M$ is a maximal Cohen--Macaulay
$R$-module, then $R\ltimes M$ is Cohen--Macaulay as well.
\end{fact}

Next, we discuss the correspondence between dualizing modules and
Gorenstein presentations given by the results of Foxby, Reiten and Sharp.

\begin{fact} \label{f0205}
Sharp~\cite[(3.1)]{sharp:gmccmlr}
showed that if $R$ is Cohen--Macaulay and a homomorphic image of a
local Gorenstein ring $Q$, then $R$ admits a dualizing module. The proof
proceeds as follows.
If $g=\depth(Q)-\depth(R)=\dim(Q)-\dim(R)$, then
$\ext^i_Q(R,Q)=0$ for $i\neq g$ and the module
$\ext^g_Q(R,Q)$ is dualizing for $R$.

The same idea gives the following.
Let $A$ be a local Cohen--Macaulay ring with a dualizing module $D$,
and assume that $R$ is Cohen--Macaulay and a module-finite $A$-algebra.
If $h=\depth(A)-\depth(R)=\dim(A)-\dim(R)$,
then $\ext^i_A(R,D)=0$ for $i\neq h$ and the module
$\ext^h_A(R,D)$ is dualizing for $R$.
\end{fact}

\begin{fact} \label{f0209}
Independently, Foxby~\cite[(4.1)]{foxby:gmarm}
and Reiten~\cite[(3)]{reiten:ctsgm}
proved the converse of Sharp's result from Fact~\ref{f0205}. Namely, they showed
that if $R$ admits a dualizing module,
then it is Cohen--Macaulay and a homomorphic image of a
local Gorenstein ring $Q$.
We sketch the proof here, as the main idea
forms the basis of our proof of Theorem~\ref{thm0001}.
See also, e.g., \cite[(3.3.6)]{bruns:cmr}.

Let $D$ be a dualizing $R$-module. It follows
from~\cite[(8.9)]{sharp:gm}
that $R$ is Cohen--Macaulay.
Set $Q=R\ltimes D$, which is Gorenstein with $\dim(Q)=\dim(R)$.
The natural surjection $\tau_D\colon Q\to R$ yields an presentation of $R$ as a homomorphic
image of the local Gorenstein ring $Q$.
\end{fact}

The next notion we need is
Auslander and Bridger's
G-dimension~\cite{auslander:adgeteac,auslander:smt}.
See also Christensen~\cite{christensen:gd}.

\begin{defn} \label{d0202}
A complex of $R$-modules
$$X=
\cdots\xra{\partial^X_{i+1}}X_{i}\xra{\partial^X_{i}}X_{i-1}\xra{\partial^X_{i-1}}\cdots
$$
is
\emph{totally acyclic} if it satisfies the following conditions:
\begin{enumerate}[\quad(1)]
\item  \label{d0202d}
Each $R$-module $X_i$ is  finitely generated and free; and
\item  \label{d0202e}
The complexes $X$ and $\Hom_R(X,R)$ are exact.
\end{enumerate}
An $R$-module $G$ is \emph{totally reflexive} if
there exists a totally acyclic complex of $R$-modules
such that $G\cong\coker(\partial^X_1)$;
in this event, the complex $X$ is a \emph{complete resolution} of $G$.
\end{defn}

\begin{fact} \label{f0299}
An $R$-module $G$ is totally reflexive if and only if
it satisfies the following:
\begin{enumerate}[\quad(1)]
\item  \label{d0202a}
The $R$-module $G$ is finitely generated;
\item  \label{d0202b}
The biduality map $\delta^R_G\colon G\to\hom_R(\hom_R(G,R),R)$, is an isomorphism; and
\item  \label{d0202c}
For all $i\geq 1$, we have $\ext^i_R(G,R)=0=\ext^i_R(\hom_R(G,R),R)$.
\end{enumerate}
See, e.g., \cite[(4.1.4)]{christensen:gd}.
\end{fact}

\begin{defn} \label{d0202'}
Let $M$ be a finitely generated $R$-module. Then $M$ has
\emph{finite G-dimension} if it has a finite resolution by totally reflexive
$R$-modules, that is, if there is an exact sequence
$$0\to G_n\to\cdots\to G_1\to G_0\to M\to 0$$
such that each $G_i$ is a totally reflexive $R$-module.
The \emph{G-dimension} of $M$, when it is finite, is the length of the shortest
finite resolution by totally reflexive
$R$-modules:
$$
\gdim_R(M)
=\inf\left\{n\geq 0\left|
\text{\begin{tabular}{c}
there is an exact sequence of $R$-modules \\
$0\to G_n\to\cdots\to  G_0\to M\to 0$ \\
such that each $G_i$ is totally reflexive \end{tabular}}
\right.\right\}.$$
\end{defn}

\begin{fact} \label{f0204}
The ring $R$ is Gorenstein if and only if every finitely generated
$R$-module has finite G-dimension; see~\cite[(1.4.9)]{christensen:gd}.
Also, the
AB formula~\cite[(1.4.8)]{christensen:gd} says that
if $M$ is a finitely generated $R$-module of finite G-dimension, then
$$\gdim_R(M)=\depth(R)-\depth_R(M).$$
\end{fact}

\begin{fact} \label{f0202}
Let $S$ be a Cohen--Macaulay local ring equipped with a
module-finite local ring homomorphism
$\tau\colon S\to R$ such that
$R$ is Cohen--Macaulay.
Then $\gdim_S(R)<\infty$
if and only if there exists an integer $g\geq 0$ such that
$\ext^i_S(R,S)=0$ for all $i\neq g$ and $\ext^g_S(R,S)$ is
a semidualizing $R$-module;
when these conditions hold, one has $g=\gdim_S(R)$.
See~\cite[(6.1)]{christensen:scatac}.

Assume that $S$ has a dualizing module $D$.
If $\gdim_S(R)<\infty$, then
$R\otimes_SD$ is a semidualizing $R$-module and
$\tor^S_i(R,D)=0$ for all $i\geq 1$;
see~\cite[(4.7),(5.1)]{christensen:scatac}.
\end{fact}

Our final background topic is Avramov and Martsinkovsky's notion of Tate cohomology~\cite{avramov:aratc}.

\begin{defn} \label{d0202''}
Let $M$ be a finitely generated $R$-module.
Considering $M$ as a complex concentrated in degree zero,
a \emph{Tate resolution} of $M$ is a diagram
of degree zero chain maps of $R$-complexes
$T\xra{\alpha}P\xra{\beta}M$
satisfying the following conditions:
\begin{enumerate}[\quad(1)]
\item
The complex $T$ is totally acyclic, and the map $\alpha_i$
is an isomorphism for $i\gg 0$;
\item
The complex $P$ is a resolution of $M$ by finitely generated free $R$-modules, and $\beta$ is the augmentation map
\end{enumerate}
\end{defn}

\begin{disc}
In~\cite{avramov:aratc}, Tate resolutions are called ``complete resolutions''.
We call them Tate resolutions in order
to avoid confusion with the terminology from Definition~\ref{d0202}.
This is consistent with~\cite{sather:tate1}.
\end{disc}

\begin{fact} \label{f0298}
By~\cite[(3.1)]{avramov:aratc}, a finitely generated $R$-module
$M$ has finite G-dimension if and only if it admits a Tate resolution.
\end{fact}

\begin{defn} \label{d0202'''}
Let $M$ be a finitely generated $R$-module of finite G-dimension,
and let
$T\xra{\alpha}P\xra{\beta}M$ be a Tate resolution of $M$.
For each integer $i$ and each $R$-module $N$, the
$i$th \emph{Tate homology} and
\emph{Tate cohomology} modules are
\begin{align*}
\tatetor^R_i(M,N)
&=\HH_i(T\otimes_RN)
&\tateext_R^i(M,N)
&=\HH_{-i}(\Hom_R(T,N)).
\end{align*}
\end{defn}

\begin{fact} \label{f0297}
Let $M$ be a finitely generated $R$-module of finite G-dimension.
For each integer $i$ and each $R$-module $N$, the
modules $\tatetor^R_i(M,N)$
and $\tateext_R^i(M,N)$
are independent of the choice of
Tate resolution of $M$,
and they are appropriately functorial in each variable by~\cite[(5.1)]{avramov:aratc}.
If $M$ has finite projective dimension, then
we have
$\tatetor^R_i(M,-)=0=\tateext_R^i(M,-)$
and
$\tatetor^R_i(-,M)=0=\tateext_R^i(-,M)$
for each integer $i$; see~\cite[(5.9) and (7.4)]{avramov:aratc}.
\end{fact}

\section{Proof of \protect{Theorem~\ref{thm0001}}} \label{sec02}

We divide the proof of Theorem~\ref{thm0001} into two pieces.
The first piece is the following result
which covers one implication.
Note that, if $\pd_Q(Q/I_1)$ or $\pd_Q(Q/I_2)$ is finite,
then condition~\eqref{thm0302d} holds automatically by
Fact~\ref{f0297}.

\begin{thm}[Sufficiency of conditions~\eqref{thm0001a}--\eqref{thm0001c} of Theorem~\ref{thm0001}]
\label{thm0302}
Let $R$ be a local Cohen--Macaulay ring with dualizing module.
Assume that there exist
a Gorenstein local ring $Q$ and
ideals $I_1,I_2\subset Q$ satisfying the following conditions:
\begin{enumerate}[\quad\rm(1)]
\item \label{thm0302a}
There is a ring isomorphism
$R\cong Q/(I_1+I_2)$;
\item \label{thm0302b}
For $j=1,2$ the quotient ring $Q/I_j$ is Cohen--Macaulay,
and $Q/I_2$ is not Gorenstein;
\item \label{thm0302d}
For all $i\in\bbz$, we have
$\tatetor^Q_i(Q/I_1,Q/I_2)=0=\tateext^i_Q(Q/I_1,Q/I_2)$;
\item \label{thm0302e}
There exists an integer $c$ such that $\ext^c_Q(Q/I_1,Q/I_2)$ is not cyclic; and
\item \label{thm0302c}
For all $i\geq 1$, we have $\tor^Q_i(Q/I_1,Q/I_2)=0$; in particular,
there is an equality $I_1\cap I_2=I_1I_2$.
\end{enumerate}
Then  $R$ admits a semidualizing module that is neither dualizing nor free.
\end{thm}

\begin{proof}
For $j=1,2$ set $R_j=Q/I_j$.
Since $Q$ is Gorenstein, we have
$\gdim_Q(R_1)<\infty$ by Fact~\ref{f0204}, so $R_1$ admits a Tate resolution
$T\xra\alpha P\xra\beta R_1$
over $Q$; see Fact~\ref{f0298}.

We claim that the induced diagram
$T\otimes_QR_2\xra{\alpha\otimes_QR_2} P\otimes_QR_2\xra{\beta\otimes_QR_2} R_1\otimes_QR_2$
is a Tate resolution of $R_1\otimes_QR_2\cong R$ over $R_2$.
The condition~\eqref{thm0302c} implies that $P\otimes_QR_2$ is a free resolution
of $R_1\otimes_QR_2\cong R$ over $R_2$, and it follows that
$\beta\otimes_QR_2$ is a quasi-isormorphism.
Of course, the complex $T\otimes_QR_2$ consists of finitely generated free $R_2$-modules,
and the map $\alpha^i\otimes_QR_2$ is an isomorphism for $i\gg 0$.
The condition $\tatetor^Q_i(R_1,R_2)=0$ from~\eqref{thm0302d}
implies that the complex $T\otimes_QR_2$ is exact. Hence, to prove the claim,
it remains to show that the first complex in the following sequence of isomorphisms is exact:
\begin{align*}
\Hom_{R_2}(T\otimes_QR_2,R_2)
&\cong\Hom_Q(T,\Hom_{R_2}(R_2,R_2))
\cong\Hom_Q(T,R_2).
\end{align*}
The isomorphisms here are given by Hom-tensor adjointness and
Hom cancellation. This explains the first step in the next sequence of isomorphisms:
$$\HH_i(\Hom_{R_2}(T\otimes_QR_2,R_2))
\cong\HH_i(\Hom_Q(T,R_2))
\cong\tateext^{-i}_Q(R_1,R_2)=0.
$$
The second step is by definition, and the third step is by assumption~\eqref{thm0302d}.
This establishes the claim.

From the claim, we conclude that $g=\gdim_{R_2}(R)$ is finite; see Fact~\ref{f0298}.
It follows from Fact~\ref{f0202} that $\ext^g_{R_2}(R,R_2)\neq 0$,
and that the $R$-module $C=\ext^g_{R_2}(R,R_2)$ is semidualizing.

To complete the proof, we need only show that $C$ is not free and not dualizing.
By assumption~\eqref{thm0302e}, the fact that $\ext^i_{R_2}(R,R_2)=0$
for all $i\neq g$ implies that $C=\ext^g_{R_2}(R,R_2)$ is not cyclic,
so $C\not\cong R$.

There is an equality of Bass series
$I^{R_2}_{R_2}(t)=t^e I^C_R(t)$ for some integer $e$.
(For instance, the vanishing $\ext^i_{R_2}(R,R_2)=0$
for all $i\neq g$ implies that there is an isomorphism
$C\simeq\shift^g\rhom_{R_2}(R,R_2)$ in $\catd(R)$,
so we can apply, e.g., \cite[(1.7.8)]{christensen:scatac}.)
By assumption~\eqref{thm0302b}, the ring $R_2$ is not Gorenstein.
Hence, the Bass series $I^{R_2}_{R_2}(t)=t^e I^C_R(t)$
is not a monomial. It follows that
the Bass series $I^C_R(t)$
is not a monomial, so
$C$ is not dualizing for $R$.
\end{proof}

The remainder of this section is devoted to the proof of the following.

\begin{thm}[Necessity of conditions~\eqref{thm0001a}--\eqref{thm0001c} of Theorem~\ref{thm0001}]
 \label{thm0303}
Let $R$ be a local Cohen--Macaulay ring with dualizing module $D$.
Assume that  $R$ admits a semidualizing module $C$ that is neither dualizing nor free.
Then there exist
a Gorenstein local ring $Q$ and
ideals $I_1,I_2\subset Q$ satisfying the following conditions:
\begin{enumerate}[\quad\rm(1)]
\item \label{thm0303a}
There is a ring isomorphism
$R\cong Q/(I_1+I_2)$;
\item \label{thm0303b}
For $j=1,2$ the quotient ring $Q/I_j$ is Cohen--Macaulay
with a dualizing module $D_j$ and is not Gorenstein;
\item \label{thm0303d}
For all $i\in\bbz$, we have
$\tatetor^Q_i(Q/I_1,Q/I_2)=0=\tateext^i_Q(Q/I_1,Q/I_2)$
and
$\tatetor^Q_i(Q/I_2,Q/I_1)=0=\tateext^i_Q(Q/I_2,Q/I_1)$;
\item \label{thm0303e}
The modules $\Hom_Q(Q/I_1,Q/I_2)$ and $\Hom_Q(Q/I_2,Q/I_1)$ are not cyclic;
\item \label{thm0303c}
For all $i\geq 1$, we have $\ext^i_Q(Q/I_1,Q/I_2)=0=\ext^i_Q(Q/I_2,Q/I_1)$
and $\tor^Q_i(Q/I_1,Q/I_2)=0$; in particular,
there is an equality $I_1\cap I_2=I_1I_2$;
\item \label{thm0303x}
For $j=1,2$ we have $\gdim_{Q/I_j}(R)<\infty$; and
\item \label{thm0303y}
There exists an $R$-module isomorphism $D_1\otimes_QD_2\cong D$,
and for all $i\geq 1$ we have $\tor^Q_i(D_1,D_2)=0$.
\end{enumerate}
\end{thm}

\begin{proof}
For the sake of readability, we include the following roadmap of the proof.

\begin{outline}
The ring $Q$ is constructed as an iterated trivial extension of $R$.
As an $R$-module, it has the form
$Q=R\oplus C\oplus \cd\oplus D$
where $\cd=\hom_R(C,D)$.
The ideals $I_j$ are then given as $I_1=0\oplus 0\oplus \cd\oplus D$
and $I_2=0\oplus C\oplus 0\oplus D$.
The details for these constructions are contained in Steps~\ref{step2} and~\ref{step5}.
Conditions~\eqref{thm0303a}, \eqref{thm0303b} and~\eqref{thm0303x} are then verified in
Lemmas~\ref{step6}--\ref{step4}. The verification of conditions~\eqref{thm0303e}
and~\eqref{thm0303c} requires more work;
it is proved in Lemma~\ref{step10}, with the help of Lemmas~\ref{step7}--\ref{step9}.
Lemma~\ref{step11} contains the verification of
condition~\eqref{thm0303y}. The proof concludes with Lemma~\ref{llast}
which contains the verification of condition~\eqref{thm0303d}.
\end{outline}

The following two steps contain notation and facts for use through the rest of the proof.

\begin{step} \label{step2}
Set $R_1=R\ltimes C$, which is Cohen--Macaulay
with $\dim(R_1)=\dim(R)$;
see Facts~\ref{f0206}
and~\ref{f0208}.
The natural injection
$\epsilon_C\colon R\to R_1$ makes $R_1$ into a module-finite $R$-algebra, so
Fact~\ref{f0205} implies that the module
$D_1=\hom_R(R_1,D)$ is dualizing for $R_1$.
There
is a sequence of $R$-module isomorphisms
$$
D_1=\hom_R(R_1,D)\cong\hom_R(R\oplus C,D)\cong\hom_R(C,D)\oplus\hom_R(R,D)
\cong \cd\oplus D.
$$
It is straightforward to show that the resulting $R_1$-module structure
on $\cd\oplus D$ is given by the following formula:
$$(r,c)(\phi,d)=(r\phi,\phi(c)+rd).$$
The kernel of the natural epimorphism $\tau_{C}\colon R_1\to R$ is the
ideal $\ker(\tau_{C})\cong 0\oplus C$.

Fact~\ref{f0209} implies that the ring
$Q=R_1\ltimes D_1$ is local and Gorenstein.
The $R$-module isomorphism in the next display is by definition:
$$Q=R_1\ltimes D_1\cong R\oplus C\oplus \cd\oplus D.$$
It is straightforward to show that the resulting ring structure on
$Q$
is given by
\begin{equation*}
(r,c,\phi,d)(r',c',\phi',d')
=(rr',rc'+r'c,r\phi'+r'\phi,\phi'(c)+\phi(c')+rd'+r'd).
\end{equation*}
The kernel of the epimorphism $\tau_{D_1}\colon Q\to R_1$ is the
ideal
$$I_1=\ker(\tau_{D_1})\cong 0\oplus 0\oplus\cd\oplus D.$$
As a $Q$-module, this is isomorphic to the $R_1$-dualizing module $D_1$.
The kernel of the composition
$\tau_C\circ\tau_{D_1}\colon Q\to R$ is the ideal
$\ker(\tau_C\tau_{D_1})\cong 0\oplus C\oplus \cd\oplus D$.

Since $Q$ is Gorenstein and $\depth(R_1)=\depth(Q)$, Fact~\ref{f0204}
implies that $R_1$ is totally reflexive as a $Q$-module.
Using the the natural
isomorphism $\Hom_Q(R_1,Q)\xra\cong(0:_QI_1)$ given by $\psi\mapsto \psi(1)$,
one shows that
the map $\Hom_Q(R_1,Q)\to I_1$ given by
$\psi\mapsto \psi(1)$ is a well-defined $Q$-module isomorphism. 
Thus $I_1$ is totally reflexive over $Q$, and it follows that
$\Hom_Q(I_1,Q)\cong R_1$.
\end{step}

\begin{step}\label{step5}
Set $R_2=R\ltimes \cd$, which is Cohen--Macaulay
with $\dim(R_2)=\dim(R)$.
The injection
$\epsilon_{\cd}\colon R\to R_2$ makes $R_2$ into a module-finite $R$-algebra, so
the module
$D_2=\hom_R(R_2,D)$ is dualizing for $R_2$.
There
is a sequence of $R$-module isomorphisms
$$
D_2=\hom_R(R_2,D)\cong\hom_R(R\oplus \cd,D)\cong\hom_R(\cd,D)\oplus\hom_R(R,D)
\cong C\oplus D.
$$
The last isomorphism is from Fact~\ref{f0203}.
The resulting $R_2$-module structure
on $C\oplus D$ is given by the following formula:
$$(r,\phi)(c,d)=(r\phi,\phi(c)+rd).$$
The kernel of the natural epimorphism $\tau_{\cd}\colon R_2\to R$ is the
ideal $\ker(\tau_{\cd})\cong 0\oplus \cd$.

The ring
$Q'=R_2\ltimes D_2$ is local and Gorenstein.
There is a sequence of $R$-module isomorphisms
$$Q'=R_2\ltimes D_2\cong R\oplus C\oplus \cd\oplus D$$
and the resulting ring structure on
$R\oplus C\oplus \cd\oplus D$
is given by
\begin{align*}
(r,c,\phi,d)(r',c',\phi',d')
=(rr',rc'+r'c,r\phi'+r'\phi,\phi'(c)+\phi(c')+rd'+r'd).
\end{align*}
That is, we have an isomorphism of rings $Q'\cong Q$.
The kernel of the epimorphism $\tau_{D_2}\colon Q\to R_2$ is the
ideal
$$I_2=\ker(\tau_{D_2})\cong 0\oplus C\oplus 0\oplus D.$$
This is isomorphic, as a $Q$-module, to the dualizing module $D_2$.
The kernel of the composition
$\tau_{\cd}\circ\tau_{D_2}\colon Q\to R$ is the ideal
$\ker(\tau_{\cd}\tau_{D_2})\cong 0\oplus C\oplus \cd\oplus D$.

As in Step~\ref{step2}, the $Q$-modules $R_2$ and
$\Hom_Q(R_2,Q)\cong I_2$ are totally reflexive, and
$\Hom_Q(I_2,Q)\cong R_2$.
\end{step}

\begin{lem}[Verification of condition~\eqref{thm0303a} from Theorem~\ref{thm0303}]
\label{step6}
With the notation of Steps~\ref{step2}--\ref{step5},
there is a ring isomorphism
$R\cong Q/(I_1+I_2)$.
\end{lem}

\begin{proof}
Consider the following sequence of $R$-module isomorphisms:
\begin{align*}
Q/(I_1+I_2)
&\cong (R\oplus C\oplus \cd\oplus D)/((0\oplus 0\oplus \cd\oplus D)+(0\oplus C\oplus 0\oplus D))\\
&\cong (R\oplus C\oplus \cd\oplus D)/(0\oplus C\oplus \cd\oplus D))\\
&\cong R.
\end{align*}
It is straightforward to check that these are ring isomorphisms.
\end{proof}

\begin{lem}[Verification of condition~\eqref{thm0303b} from Theorem~\ref{thm0303}]
\label{step3}
With the notation of Steps~\ref{step2} and~\ref{step5},
each ring $R_j\cong Q/I_j$ is Cohen--Macaulay
with a  dualizing module $D_j$ and is not Gorenstein.
\end{lem}

\begin{proof}
It remains only to show that each ring $R_j$ is not Gorenstein,
that is, that $D_j$ is not isomorphic to $R_j$ as an $R_j$-module.

For $R_1$, suppose by way of contradiction
that there is an $R_1$-module isomorphism $D_1\cong R_1$.
It follows that this is an $R$-module isomorphism via the natural
injection $\epsilon_C\colon R\to R_1$.
Thus, we have $R$-module isomorphisms
$$\cd\oplus D\cong D_1\cong R_1\cong R\oplus C.$$
Computing minimal numbers of generators, we have
\begin{align*}
\mu_R(\cd)+\mu_R(D)
&=\mu_R(\cd\oplus D)
=\mu_R(R\oplus C)
=\mu_R(R)+\mu_R(C)\\
&=1+\mu_R(C)
\leq 1+\mu_R(C)\mu_R(\cd)
=1+\mu_R(D).
\end{align*}
The last step in this sequence follows from Fact~\ref{f0203}.
It follows that $\mu_R(\cd)=1$, that is, that $\cd$ is cyclic.
From the isomorphism $R\cong \Hom_R(C,C)$, one concludes
that $\ann_R(C)=0$, and hence
$\cd\cong R/\ann_R(\cd)\cong R$.
It follows that
$$C\cong\hom_R(\cd,D)\cong\hom_R(R,D)\cong D$$
contradicting the assumption that $C$ is not dualizing for $R$.
(Note that this uses the uniqueness statement from Fact~\ref{f0206}.)

Next,  observe that $\cd$ is not free and is not dualizing
for $R$; this follows from the isomorphism
$C\cong\hom_R(\cd,D)$ contained in Fact~\ref{f0203}, using the
assumption that $C$ is not free and not dualizing.
Hence, the proof that $R_2$ is not Gorenstein follows as in the previous paragraph.
\end{proof}

\begin{lem}[Verification of condition~\eqref{thm0303x} from Theorem~\ref{thm0303}]
\label{step4}
With the notation of Steps~\ref{step2}--\ref{step5},
we have $\gdim_{R_j}(R)=0$ for $j=1,2$.
\end{lem}

\begin{proof}
To show that $\gdim_{R_1}(R)=0$, it suffices
to show that $\ext^i_{R_1}(R,R_1)=0$ for all $i\geq 1$ and that
$\hom_{R_1}(R,R_1)\cong C$; see Fact~\ref{f0202}.
To this end, we note that there are isomorphisms of $R$-modules
$$\hom_R(R_1,C)\cong\hom_R(R\oplus C,C)
\cong\hom_R(C,C)\oplus\hom_R(R,C)\cong R\oplus C\cong R_1$$
and it is straightforward to check that the composition
$\hom_R(R_1,C)\cong R_1$ is an $R_1$-module isomorphism.
Furthermore, for $i\geq 1$ we have
$$\ext^i_R(R_1,C)\cong\ext^i_R(R\oplus C,C)
\cong\ext^i_R(C,C)\oplus\ext^i_R(R,C)=0.$$
Let $I$ be an injective resolution of $C$ as an $R$-module.
The previous two displays imply that
$\hom_R(R_1,I)$ is an injective resolution of $R_1$ as an $R_1$-module.
Using the 
fact that the composition $R\xra{\epsilon_C}R_1\xra{\tau_C}R$
is the identity $\id_R$,
we conclude that
$$\hom_{R_1}(R,\hom_R(R_1,I))\cong\hom_{R}(R\otimes_{R_1}R_1,I)
\cong\hom_R(R,I)\cong I$$
and hence
$$\ext^i_{R_1}(R,R_1)\cong\HH^i(\hom_{R_1}(R,\hom_R(R_1,I)))\cong\HH^i(I)
\cong\begin{cases} 0 & \text{if $i\geq 1$} \\ C & \text{if $i=0$} \end{cases}$$
as desired.\footnote{Note that the finiteness of $\gdim_{R_1}(R)$ can also be deduced
from~\cite[(2.16)]{holm:smarghd}.}

The proof for $R_2$ is similar.
\end{proof}

The next three results are for the proof of Lemma~\ref{step10}.

\begin{lem}\label{step7}
With the notation of Steps~\ref{step2} and~\ref{step5},
one has $\tor^R_i(R_1,R_2)=0$ for all $i\geq 1$,
and  there is an $R_1$-algebra isomorphism $R_1\otimes_R R_2\cong Q$.
\end{lem}

\begin{proof}
The Tor-vanishing comes from  the following sequence of
$R$-module isomorphisms
\begin{align*}
\tor^R_i(R_1,R_2)
&\cong\tor^R_i(R\oplus C,R\oplus \cd) \\
&\cong\tor^R_i(R,R) \oplus\tor^R_i(C,R) \oplus\tor^R_i(R,\cd) \oplus\tor^R_i(C,\cd) \\
&\cong\begin{cases}
R\oplus C\oplus\cd\oplus D & \text{if $i=0$} \\
0 & \text{if $i\neq 0$.} \end{cases}
\end{align*}
The first isomorphism is by definition;
the second isomorphism is elementary;
and the third isomorphism is from Fact~\ref{f0203}.

Moreover, it is straightforward to verify  that in the case $i=0$ the isomorphism
$R_1\otimes_R R_2\cong Q$ has the form
$\alpha\colon R_1\otimes_R R_2\xra\cong Q$
given by
$$(r,c)\otimes(r',\phi')\mapsto(rr',r'c,r\phi',\phi'(c)).$$
It is routine to check that this is a ring homomorphism,
that is, a ring isomorphism.
Let $\xi\colon R_1\to R_1\otimes_RR_2$ be given by
$(r,c)\mapsto (r,c)\otimes (1,0)$. Then one has
$\alpha\xi=\epsilon_{D_1}\colon R_1\to Q$.
It follows that $R_1\otimes_R R_2\cong Q$ as an $R_1$-algebra.
\end{proof}

\begin{lem}\label{step8}
Continue with the notation of Steps~\ref{step2} and~\ref{step5}.
In the tensor product $R\otimes_{R_1}Q$ we have
$1\otimes(0,c,0,d)=0$
for all $c\in C$ and all $d\in D$.
\end{lem}

\begin{proof}
Recall that Fact~\ref{f0203} implies that
the evaluation map $C\otimes_R\cd\to D$ given by $c'\otimes\phi\mapsto\phi(c')$
is an isomorphism. Hence, there exist $c'\in C$ and $\phi\in\cd$ such that
$d=\phi(c')$.
This explains the first equality in the  sequence
\begin{equation} \label{disc0201b}
\begin{split}
1\otimes(0,0,0,d)
&=1\otimes(0,0,0,\phi(c'))
=1\otimes[(0,c')(0,0,\phi,0)]\\
&=[1(0,c')]\otimes(0,0,\phi,0)
=0\otimes(0,0,\phi,0)
=0.
\end{split}
\end{equation}
The second equality is by definition of the $R_1$-module structure on $Q$;
the third equality is from the fact that we are tensoring over $R_1$;
the fourth equality is from the fact that the $R_1$-module structure on $R$ comes
from the natural surjection $R_1\to R$, with the fact that $(0,c)\in 0\oplus C$ which is the
kernel of this surjection.

On the other hand, using similar reasoning, we have
\begin{equation} \label{disc0201c}
\begin{split}
1\otimes(0,c,0,0)
&=1\otimes[(0,c)(1,0,0,0)]
=[1(0,c)]\otimes(1,0,0,0)\\
&=0\otimes(1,0,0,0)
=0.
\end{split}
\end{equation}
Combining~\eqref{disc0201b} and~\eqref{disc0201c} we have
$$1\otimes(0,c,0,d)
=[1\otimes(0,0,0,d)]+[1\otimes(0,c,0,0)]
=0$$
as claimed.
\end{proof}

\begin{lem}\label{step9}
With the notation of Steps~\ref{step2} and~\ref{step5},
one has $\tor^{R_1}_i(R,Q)=0$ for all $i\geq 1$,
and there is a $Q$-module isomorphism $R\otimes_{R_1}Q\cong R_2$.
\end{lem}

\begin{proof}
Let $P$ be an $R$-projective resolution of $R_2$.
Lemma~\ref{step7} implies that $R_1\otimes_RP$ is a projective
resolution of $R_1\otimes_RR_2\cong Q$ as an $R_1$-module.
From the following sequence of isomorphisms
$$R\otimes_{R_1}(R_1\otimes_RP)
\cong (R\otimes_{R_1}R_1)\otimes_RP
\cong R\otimes_RP
\cong P$$
it follows that, for $i\geq 1$, we have
$$\tor^{R_1}_i(R,Q)
\cong\HH_i(R\otimes_{R_1}(R_1\otimes_RP))
\cong\HH_i(P)
=0$$
where the final vanishing comes from the assumption that $P$ is a resolution
of a module and $i\geq 1$.

This reasoning shows that there is an $R$-module isomorphism
$\beta\colon R_2\xra\cong R\otimes_{R_1}Q$. This isomorphism is
equal to the composition
$$R_2
\xra\cong R\otimes_R R_2
\xra\cong R\otimes_{R_1}(R_1\otimes_RR_2)
\xra[R\otimes_{R_1}\alpha]{\cong} R\otimes_{R_1}Q
$$
and is therefore given by
\begin{equation}\label{eqz}
(r,\phi)\mapsto 1\otimes(r,\phi)\mapsto 1\otimes[(1,0)\otimes(r,\phi)]\mapsto 1\otimes(r,0,\phi,0).
\end{equation}
We claim that $\beta$ is a $Q$-module isomorphism.
Recall that
the $Q$-module structure on $R_2$ is given via the natural surjection $Q\to R_2$,
and so is described as
$$(r,c,\phi,d)(r',\phi')=(r,\phi)(r',\phi')=(rr',r\phi'+r'\phi).$$
This explains the first equality in the following sequence
\begin{align*}
\beta((r,c,\phi,d)(r',\phi'))
&=\beta(rr',r\phi'+r'\phi)
=1\otimes (rr',0,r\phi'+r'\phi,0).
\end{align*}
The second equality is by~\eqref{eqz}.
On the other hand, the definition of $\beta$ explains
the first equality in the  sequence
\begin{align*}
(r,c,\phi,d)\beta(r',\phi')
&=(r,c,\phi,d)[1\otimes(r',0,\phi',0)]\\
&=1\otimes[(r,c,\phi,d)(r',0,\phi',0)]\\
&=1\otimes(rr',r'c,r\phi'+r'\phi,r'd+\phi'(c))\\
&=[1\otimes(rr',0,r\phi'+r'\phi,0)]+[1\otimes(0,r'c,0,r'd+\phi'(c))]\\
&=1\otimes(rr',0,r\phi'+r'\phi,0).
\end{align*}
The second equality is from the definition of the $Q$-modules structure on
$R\otimes_{R_1}Q$;
the third equality is from the definition of the multiplication in $Q$;
the fourth equality is by bilinearity;
and the fifth equality is by Lemma~\ref{step8}.
Combining these two sequences, we conclude that $\beta$ is a $Q$-module isomorphism,
as claimed.
\end{proof}

\begin{lem}[Verification of conditions~\eqref{thm0303e}--\eqref{thm0303c} from Theorem~\ref{thm0303}]
\label{step10}
With the notation of Steps~\ref{step2} and~\ref{step5},
the modules $\Hom_Q(R_1,R_2)$ and $\Hom_Q(R_2,R_1)$ are not cyclic.
Also, one has $\ext^i_Q(R_1,R_2)=0=\ext^i_Q(R_2,R_1)$
and $\tor^Q_i(R_1,R_2)=0$ for all $i\geq 1$; in particular,
there is an equality $I_1\cap I_2=I_1I_2$.
\end{lem}

\begin{proof}
Let $L$ be a projective resolution of $R$ over $R_1$.
Lemma~\ref{step9} implies that the complex $L\otimes_{R_1}Q$ is a projective
resolution of $R\otimes_{R_1}Q\cong R_2$ over $Q$.
We have isomorphisms
$$(L\otimes_{R_1}Q)\otimes_QR_1
\cong L\otimes_{R_1}(Q\otimes_QR_1)
\cong L\otimes_{R_1}R_1
\cong L$$
and it follows that, for $i\geq 1$, we have
$$\tor^Q_i(R_2,R_1)
\cong\HH_i((L\otimes_{R_1}Q)\otimes_QR_1)
\cong\HH_i(L)=0$$
since $L$ is a projective resolution.

The equality $I_1\cap I_2=I_1 I_2$
follows from the  direct computation
$$
I_1\cap I_2
=(0\oplus 0\oplus \cd\oplus D)\cap(0\oplus C\oplus 0\oplus D)
=0\oplus 0\oplus 0\oplus D=I_1I_2$$
or from the  sequence
$(I_1\cap I_2)/(I_1 I_2)\cong\tor^Q_1(Q/I_1,Q/I_2)=0$.

Let $P$ be a projective resolution of $R_1$ over $Q$. From the fact that $\tor_i^Q(R_2,R_1)=0$ for all $i\ge 1$ we get that $P\otimes_QR_2$ is a projective
resolution of $R$ over $R_2$.  Since the complexes $\Hom_Q(P,R_2)$ and $\Hom_{R_2}(P\otimes_QR_2,R_2)$ are isomorphic, we therefore have the isomorphisms
\[
\ext^i_Q(R_1,R_2)\cong\ext^i_{R_2}(R,R_2)
\]
for all $i\ge 0$.
By the fact that $\gdim_{R_2}(R)=0$, we conclude that
\begin{align*}
\ext^i_{Q}(R_1,R_2)
&\cong\begin{cases}
\cd & \text{if $i=0$} \\
0 & \text{if $i\neq 0$.}
\end{cases}
\end{align*}
Since $C$ is not dualizing, the module
$\Hom_Q(R_1,R_2)\cong\ext^0_{Q}(R_1,R_2)\cong\cd$ is not cyclic.

The verification for $\Hom_Q(R_2,R_1)$ and $\ext^i_{Q}(R_2,R_1)$
is similar.
\end{proof}

\begin{lem}[Verification of condition~\eqref{thm0303y} from Theorem~\ref{thm0303}]
\label{step11}
With the notation of Steps~\ref{step2} and~\ref{step5},
there is an $R$-module isomorphism $D_1\otimes_QD_2\cong D$,
and for all $i\geq 1$ we have $\tor^Q_i(D_1,D_2)=0$.
\end{lem}

\begin{proof}
There is a short exact sequence of $Q$-module homomorphisms
$$0\to D_1\to Q\xra{\tau_{D_1}}R_1\to 0.$$
For all $i\geq 1$, we have
$\tor^Q_i(Q,R_2)=0=\tor^Q_i(R_1,R_2)$,
so the long exact sequence in $\tor^Q_i(-,R_2)$ associated to the
displayed sequence implies that
$\tor^Q_i(D_1,R_2)=0$ for all $i\geq 1$.
Consider the next short exact sequence of $Q$-module homomorphisms
$$0\to D_2\to Q\xra{\tau_{D_2}}R_2\to 0.$$
The associated long exact sequence in
$\tor^Q_i(D_1,-)$ implies that
$\tor^Q_i(D_1,D_2)=0$ for all $i\geq 1$.

It is straightforward to verify the following sequence of $Q$-module isomorphisms
\begin{align*}
R\otimes_{R_1}D_1
&\cong \left(\frac{R\ltimes C}{0\oplus C}\right)\otimes_{R\ltimes C}(C^{\dagger}\oplus D)
\cong \frac{C^{\dagger}\oplus D}{(0\oplus C)(C^{\dagger}\oplus D)}
\cong \frac{C^{\dagger}\oplus D}{0\oplus D}
\cong C^{\dagger}
\end{align*}
and similarly
$$R\otimes_{R_2}D_2
\cong C.$$
These combine to explain the third isomorphism in the following sequence:
\begin{align*}
D_1\otimes_QD_2
&\cong R\otimes_Q(D_1\otimes_QD_2)\cong (R\otimes_QD_1)\otimes_R(R\otimes_QD_2)
\cong C^\dagger\otimes_RC
\cong D.
\end{align*}
For the first isomorphism, use the fact that $D_j$ is annihilated by $D_j=I_j$ for $j=1,2$
to conclude that $D_1\otimes_QD_2$ is annihilated by $I_1+I_2$;
it follows that $D_1\otimes_QD_2$ is naturally a module over the quotient
$Q/(I_1+I_2)\cong R$.
The second isomorphism is standard, and the fourth one is from
Fact~\ref{f0203}.
\end{proof}

\begin{lem}[Verification of condition~\eqref{thm0303d} from Theorem~\ref{thm0303}]
\label{llast}
With the notation of Steps~\ref{step2}--\ref{step5},
we have
$\tatetor^Q_i(R_1,R_2)=0=\tateext^i_Q(R_1,R_2)$
and
$\tatetor^Q_i(R_2,R_1)=0=\tateext^i_Q(R_2,R_1)$
for all $i\in\bbz$.
\end{lem}

\begin{proof}
We verify that
$\tatetor^Q_i(R_1,R_2)=0=\tateext^i_Q(R_1,R_2)$.
The proof of the other vanishing is similar.

Recall from Step~\ref{step2} that $R_1$ is  totally reflexive
as a $Q$-module.  
We construct a complete resolution of $R_1$ over $Q$
by splicing a minimal $Q$-free resolution $P$ of $R_1$ with its dual $P^*=\Hom_Q(P,Q)$.
Using the fact that $R_1^*$ is isomorphic to $ I_1$, the first syzygy of $R_1$ in $P$,
we conclude that $X^*\cong X$.
This explains the second isomorphism in the next sequence
wherein $i$ is an arbitrary integer:
\begin{equation}\label{iso23}
\begin{split}
\tatetor^Q_i(R_1,R_2)
&\cong\HH_i(X\otimes_QR_2)
\cong\HH_i(X^*\otimes_QR_2)\\
&\cong\HH_i(\Hom_Q(X,R_2))
\cong \tateext^{-i}_Q(R_1,R_2).
\end{split}
\end{equation}
The third isomorphism is standard,
since each $Q$-module $X_i$ is finitely generated and free, and the other
isomorphisms are by definition.

For $i\geq 1$, the complex $X$ provides the
second  steps in the next displays:
\begin{align*}
\tateext^{-i}_Q(R_1,R_2)
\cong\tatetor^Q_i(R_1,R_2)
&\cong\tor^Q_i(R_1,R_2) =0 \\
\tatetor^Q_{-i}(R_1,R_2)
\cong\tateext^i_Q(R_1,R_2)
&\cong\ext^i_Q(R_1,R_2) =0.
\end{align*}
The first  steps are from~\eqref{iso23}, and the third
 steps are from Lemma~\ref{step10}.

To complete the proof it suffices by~\eqref{iso23} to show that
$\tateext^0_Q(R_1,R_2) =0$.
For this, we recall the exact sequence
\begin{equation*}
0\to
\Hom_Q(R_1,Q)\otimes_QR_2\xra{\nu}
\Hom_Q(R_1,R_2)\to
\tateext^0_Q(R_1,R_2)\to 0
\end{equation*}
from~\cite[(5.8(3))]{avramov:aratc}.
Note that this uses the fact that $R_1$ is totally reflexive as a $Q$-module,
with the condition $\tateext^{-1}_Q(R_1,R_2)=0$ which we have already verified.
Also, the map $\nu$ is given by the formula
$\nu(\psi\otimes r_2)=\psi_{r_2}\colon R_1\to R_2$
where $\psi_{r_2}(r_1)=\psi(r_1)r_2$.
Thus, to complete the proof, we need only show that the map
$\nu$ is surjective.

As with the isomorphism $\alpha\colon\Hom_Q(R_1,Q)\xra\cong I_1$,
it is straightforward to show that the map $\beta\colon\Hom_Q(R_1,R_2) \to\cd$ given
by $\phi\mapsto\phi(1)$ is a well-defined $Q$-module isomorphism.
Also, from Lemma \ref{step10} we have that $I_1I_2=0\oplus0\oplus0\oplus D$,
considered as a subset of $I_1=0\oplus0\oplus\cd\oplus D\subset
R\oplus C\oplus \cd\oplus D=Q$. In particular,
the map $\sigma\colon I_1/I_1I_2\to \cd$ given by
$\ol{(0,0,f,d)}\mapsto f$ is a well-defined $Q$-module isomorphism.

Finally, it is straightforward to show that the following diagram commutes:
$$
\xymatrix{
\Hom_Q(R_1,Q)\otimes_QR_2
\ar[d]_{\alpha\otimes_QR_2}^-{\cong}
\ar[rrr]^-{\nu}
&&&\Hom_Q(R_1,R_2)
\ar[d]_{\beta}^-{\cong} \\
I_1\otimes_QR_2
\ar[r]^-=
&I_1\otimes_QQ/I_2
\ar[r]^-{\delta}_-{\cong}
&I_1/I_1I_2
\ar[r]^-{\sigma}_-{\cong}
& \cd.}$$
From this, it follows that $\nu$ is surjective, as desired.
\end{proof}

This completes the proof of Theorem~\ref{thm0303}.
\end{proof}

\section{Constructing Rings with Non-trivial Semidualizing Modules} \label{sec99}

We begin this section with the following application of Theorem~\ref{thm0302}.

\begin{prop} \label{thm9901}
Let $R_1$ be a local Cohen--Macaulay ring with dualizing module $D_1\not\cong R_1$
and $\dim(R_1)\geq 2$.
Let $\x=x_1,\ldots,x_n\in R_1$ be an $R_1$-regular sequence with $n\geq 2$,
and fix an integer $t\geq 2$. Then the ring $R=R_1/(\x)^t$ has a semidualizing module
$C$ that is neither dualizing no free.
\end{prop}

\begin{proof}
We verify the conditions~\eqref{thm0302a}--\eqref{thm0302c} from
Theorem~\ref{thm0302}.

\eqref{thm0302a}
Set $Q=R_1\ltimes D_1$ and $I_1=0\oplus D_1\subset Q$.
Consider the elements
$y_i=(x_i,0)\in Q$ for $i=1,\ldots,n$.
It is straightforward to show that
the sequence $\y=y_1,\ldots,y_n$ is $Q$-regular.
With $R_2=Q/(\y)^t$, we have
$R\cong R_1\otimes_QR_2$.
That is, with $I_2=(\y)^t$, condition~\eqref{thm0302a}
from Theorem~\ref{thm0302} is satisfied.

\eqref{thm0302b}
The assumption $D_1\not\cong R_1$ implies that $R_1$ is not Gorenstein.
It is well-known that $\type(R_2)=\binom{t+n-2}{n-1}>1$, so $R_2$ is
not Gorenstein.

\eqref{thm0302d}
By Fact~\ref{f0297}, it suffices to show that
$\pd_Q(R_2)<\infty$.
Since $\y$ is a $Q$-regular sequence,
the associated graded ring $\oplus_{i=0}^\infty (\y)^i/(\y)^{i+1}$ is isomorphic
as a $Q$-algebra to the polynomial ring
$Q/(\y)[Y_1,\ldots,Y_n]$. It follows that the $Q$-module $R_2\cong Q/(\y)^t$
has a finite filtration $0=N_r\subset N_{r-1}\subset\cdots\subset N_0=R_2$
such that $N_{i-1}/N_i\cong Q/(\y)$ for $i=1,\ldots,r$.
Since each quotient $N_{i-1}/N_i\cong Q/(\y)$ has finite projective dimension over $Q$,
the same is true for $R_2$.

\eqref{thm0302e}
The following isomorphisms are straightforward to verify:
$$
R_2
=Q/(\y)^t
\cong [R_1/(\x)^t]\ltimes [D_1/(\x)^tD_1]
\cong R\ltimes [D_1/(\x)^tD_1].$$
Since $\x$ is $R_1$-regular, it is also $D_1$-regular.
Using this, one checks readily that
\begin{align*}
\Hom_Q(R_1,R_2)
&\cong \{z\in R_2\mid I_1z=0\}
=0\oplus [D_1/(\x)^tD_1].
\end{align*}
Since $D_1$ is not cyclic and $\x$ is contained in the maximal ideal of $R_1$,
we conclude that $\Hom_Q(R_1,R_2)\cong D_1/(\x)^tD_1$ is not cyclic.

\eqref{thm0302c}
The $Q$-module $R_1$ is totally reflexive; see Facts~\ref{f0204}--\ref{f0202}.
It follows from~\cite[(2.4.2(b))]{christensen:gd} that
$\tor^Q_i(R_1,N)=0$ for all $i\geq 1$ and for all $Q$-modules $N$ of finite flat dimension;
see also~\cite[(4.13)]{auslander:smt}. Thus, we have
$\tor^Q_i(R_1,R_2)=0$ for all $i\geq 1$.
\end{proof}

\begin{disc}
One can use the results of~\cite{avramov:rhafgd} directly to show that the ring
$R$ in Proposition~\ref{thm9901} has a non-trivial semidualizing module.
(Specifically, the relative dualizing module of the natural surjection $R_1\to R$
works.) However, our proof illustrates the concrete criteria of Theorem~\ref{thm0302}.
\end{disc}

We conclude by showing that there exists a Cohen--Macaulay local ring $R$
that does not admit a dualizing module and does admit a semidualizing module
$C$ such that $C\not\cong R$. The construction is essentially
from~\cite[p.~92, Example]{vasconcelos:dtmc}.

\begin{ex} \label{exlast}
Let $A$ be a local Cohen--Macaulay ring that does not admit a dualizing
module. (Such rings are known to exist by a result of Ferrand and
Raynaud~\cite{ferrand}.)
Set $R=A[X,Y]/(X,Y)^2\cong A\ltimes A^2$ and consider the $R$-module
$C=\Hom_A(R,A)$. Since $R$ is finitely generated and free as an $A$-module,
Fact~\ref{f0202} shows that $C$ is a semidualizing $R$-module.
The composition of the natural inclusion
$A\to R$ and the natural surjection $R\to A$ is the identity on $A$.

If $R$ admitted a dualizing module $D$, then the module
$\Hom_R(A,D)$ would be  a dualizing $A$-module by
Fact~\ref{f0205}, contradicting our assumption on $A$.
(Alternately, since $A$ is not a homomorphic image of a Gorenstein ring,
we conclude from the surjection $R\to A$ that $R$ is not a homomorphic image of a Gorenstein ring.)

We show that $C\not\cong R$. It suffices to show that
$\Hom_R(A,C)\not\cong\Hom_R(A,R)$.
We compute:
\begin{gather*}
\Hom_R(A,C)
\cong\Hom_R(A,\Hom_A(R,A))
\cong\Hom_A(R\otimes_RA,A)
\cong\Hom_A(A,A)
\cong A
\\
\Hom_R(A,R)
\cong\{r\in R\mid (0\oplus A^2)r=0\}
=0\oplus A^2
\cong A^2
\end{gather*}
which gives the  desired conclusion.
\end{ex}

\section*{Acknowledgments}
We are grateful to Lars W.\ Christensen for helpful comments about this work.


\providecommand{\bysame}{\leavevmode\hbox to3em{\hrulefill}\thinspace}
\providecommand{\MR}{\relax\ifhmode\unskip\space\fi MR }
\providecommand{\MRhref}[2]{%
  \href{http://www.ams.org/mathscinet-getitem?mr=#1}{#2}
}
\providecommand{\href}[2]{#2}

\end{document}